\documentclass{book}
\title{Markov's Theorem}
\author{Matteo Barucco\\ Nirvana Coppola}
\date{\today}
\usepackage{amsfonts}
\usepackage{amsmath}
\usepackage{amssymb}
\usepackage{amstext}
\usepackage{amsthm}

\usepackage[centering]{geometry}

\usepackage{mathtools}
\usepackage{mathrsfs}

\usepackage[english]{babel}
\usepackage[utf8]{inputenc}
\usepackage[T1]{fontenc}
\usepackage{lmodern}

\usepackage{enumerate} 
\usepackage{subfiles}

\usepackage{nicefrac} 
\usepackage{textcomp}
\usepackage{graphicx}
\usepackage{color}
\usepackage{transparent}
\usepackage{caption}
\usepackage{microtype} 
\usepackage{afterpage} 
\usepackage{subfig} 
\DeclareGraphicsExtensions{.pdf,.png,.jpg} 
\usepackage{setspace} 
\usepackage{emptypage}
\usepackage{eso-pic} 

\usepackage{algorithm}  
\usepackage{algpseudocode}
\makeatletter
\def\BState{\State\hskip-\ALG@thistlm}
\makeatother

\usepackage{booktabs} 
\usepackage{numprint} 
\usepackage{tabularx} 
\usepackage{siunitx} 

\usepackage{fancyhdr} 

\fancypagestyle{plain}{%
\fancyhf{} 
\fancyfoot[C]{\thepage} 

}

\usepackage{hyperref}

\allowdisplaybreaks[1]

\newcommand{\nuovapagina}{
  \newpage
  \thispagestyle{empty}
  \cleardoublepage
}

{
   \end{minipage}
   \vspace*{\stretch{6}}
   \nuovapagina
}

\graphicspath{{immagini/}} 
\theoremstyle{definition}
\newtheorem{theorem}{Theorem}
\newtheorem{lemma}[theorem]{Lemma}
\newtheorem{proposition}[theorem]{Proposition}

\newtheorem{fact}[theorem]{Fact}
 
\theoremstyle{definition}
\newtheorem{definition}[theorem]{Definition}
\newtheorem{remark}[theorem]{Remark}

\makeatletter 
\renewenvironment{proof}[1][\proofname]{%
  \par\pushQED{\qed}\normalfont%
  \topsep6\p@\@plus6\p@\relax
  \trivlist\item[\hskip\labelsep\bfseries#1\@addpunct{.}]%
  \ignorespaces
}{%
  \popQED\endtrivlist\@endpefalse
}
\makeatother

\def\bigquotient#1#2{%
    \raise1ex\hbox{$#1$}\Big/\lower1ex\hbox{$#2$}%
}

\def\quotient#1#2{\mathchoice
{\raisebox{.3ex}{$\mathsurround=0pt\displaystyle #1$}\mkern -1mu/\mkern -1mu\raisebox{-.5ex}{$\mathsurround=0pt\displaystyle #2$}}
{\raisebox{.3ex}{$\mathsurround=0pt\textstyle #1$}\mkern -1mu/\mkern -1mu\raisebox{-.5ex}{$\mathsurround=0pt\textstyle #2$}}
{\raisebox{.1ex}{$\mathsurround=0pt\scriptstyle #1$}\mkern -1mu/\mkern -1mu\raisebox{-.3ex}{$\mathsurround=0pt\scriptstyle #2$}}
{\raisebox{.1ex}{$\mathsurround=0pt\scriptscriptstyle #1$}\mkern -1mu/\mkern -1mu\raisebox{-.1ex}{$\mathsurround=0pt\scriptscriptstyle #2$}}}

\DeclareMathOperator{\inter}{int}

\begin{document}

\maketitle
\tableofcontents


\chapter{Introduction}

This survey consists of a detailed proof of Markov's Theorem based on \cite{B} and Carlo Petronio's classes. It was part of an exam project in A.Y. 2016/2017 for the course Knot Theory.

Let us first introduce some basic definitions and known properties about braids.

\begin{definition}
A $n$-braid or braid over $n$ strings is the image of a function:
$$\alpha : \coprod_{j=1}^n [0,1] \rightarrow \mathbb{R}^2 \times [0,1]$$
such that, if $\alpha_j$ is the restriction of $\alpha$ on the $j$-th copy of $[0,1]$, there exists a permutation $\tau \in S_n$ with the following properties:
\begin{align*}
&\alpha^{(j)}(0)=(0,0,j);\\
&\alpha^{(j)}(1)=(0,0,\tau(j));\\
&\alpha^{(j)}_3 \text{ is strictly increasing for every } j;\\
&\alpha \text{ is a proper embedding.}
\end{align*}
\end{definition}

\begin{figure}[htb]
\centering
\includegraphics{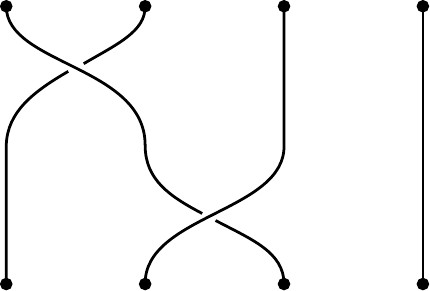}
\caption{A $4$-braid $\beta$}
\label{beta}
\end{figure}

\begin{definition}
Two $n$-braids $\alpha_0$ and $\alpha_1$ are equivalent if the maps defining them are isotopic through a family of maps with the same properties.
Note that two equivalent braids have the same associated permutation.
\end{definition}

\begin{definition}
The group $B_n$ of $n$-braids is the set of the equivalence classes of $n$-braids as defined above, with the operation of juxtaposition. Given two $n$-braids $\alpha$ and $\beta$, their juxtaposition is
\begin{align*}
\alpha \ast \beta^{(j)}_i (t) = \left\lbrace
\begin{array}{ll}
\alpha^{(j)}_i (2t) &\text{ if } 0 \leq t \leq \quotient{1}{2}\\
\beta^{(\tau(j))}_i (2t-1) &\text{ if } \quotient{1}{2} \leq t \leq 1
\end{array}
\right. \text{ for } i=1,2;\\
\alpha \ast \beta^{(j)}_3 (t) = \left\lbrace
\begin{array}{ll}
\quotient{1}{2} \ \alpha^{(j)}_3 (2t) &\text{ if } 0 \leq t \leq \quotient{1}{2}\\
\quotient{1}{2} +\quotient{1}{2}\ \beta^{(\tau(j))}_3 (2t-1) &\text{ if } \quotient{1}{2} \leq t \leq 1.
\end{array}
\right.
\end{align*}

\end{definition}

\begin{definition}
The closure of a braid is the link obtained by connecting the corresponding ends through concentric arcs (Figure \ref{fig_chiusura}).
\end{definition}

\begin{figure}[htb]
\centering
\includegraphics{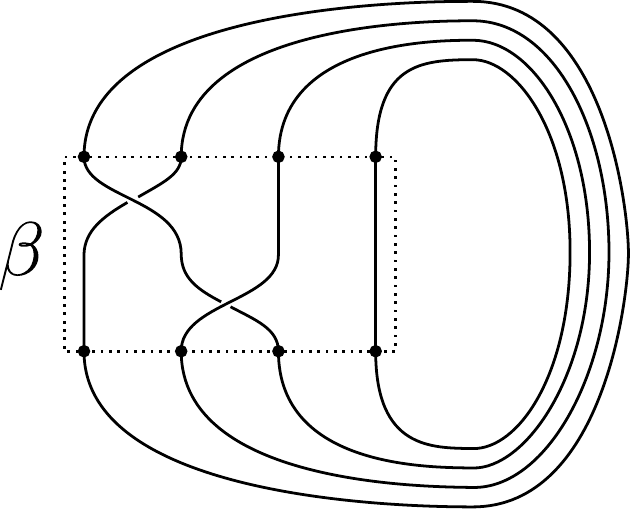}
\caption{The closure $\hat{\beta}$ of the braid $\beta$ of Figure \ref{beta}}
\label{fig_chiusura}
\end{figure}

\begin{fact}
The group of $n$-braids has the following presentation:
\begin{align*}
B_n=\left\langle \sigma_1, \dots \sigma_{n-1} | \sigma_i \sigma_j = \sigma_j \sigma_i \text{ if } |i-j| \geq 2,\ \sigma_i \sigma_{i+1} \sigma_i = \sigma_{i+1} \sigma_i \sigma_{i+1} \right\rangle,
\end{align*}
where the generator $\sigma_i$ is shown in Figure \ref{fig_sigmai} and its inverse is in Figure \ref{sigmaiinv}.
\end{fact}

\begin{remark}
The first kind of relation in the presentation of the $n$-braids group is easy to understand, and we can notice that the second kind of relation translates into a Reidemeister move $R_{III}$, as shown in figure \ref{R3}.
\end{remark}
\newpage 

\begin{figure}[htb]
\centering
\includegraphics{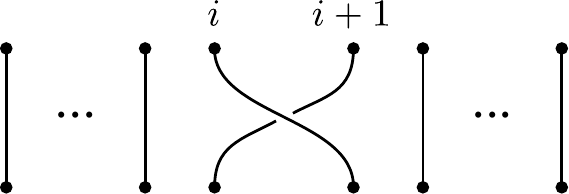}
\caption{The generator $\sigma_i$}
\label{fig_sigmai}
\end{figure}

\begin{figure}[htb]
\centering
\includegraphics{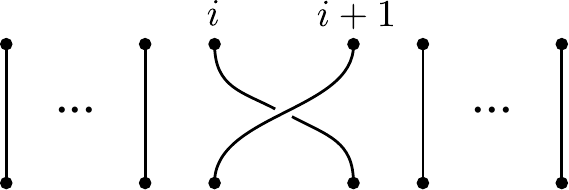}
\caption{The inverse $\sigma_i^{-1}$ of $\sigma_i$}
\label{sigmaiinv}
\end{figure}

\begin{figure}[htb]
\centering
\includegraphics{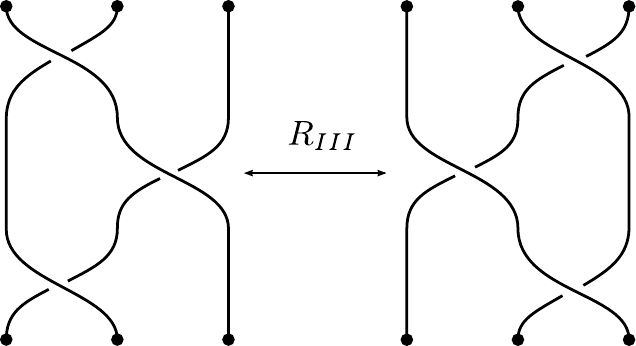}
\caption{The relation $\sigma_i \sigma_{i+1} \sigma_i = \sigma_{i+1} \sigma_i \sigma_{i+1}$}
\label{R3}
\end{figure}

\newpage
We now "algebraically" define two deformation moves that are crucial for Markov's Theorem.

\begin{definition}
Conjugation move (Figure \ref{fig_coniugio}):
\begin{align*}
B_m \ni c \mapsto w^{-1} c w \in B_m\text{, where }w \in B_m.
\end{align*}
\end{definition}

\begin{figure}[htb]
\centering
\includegraphics{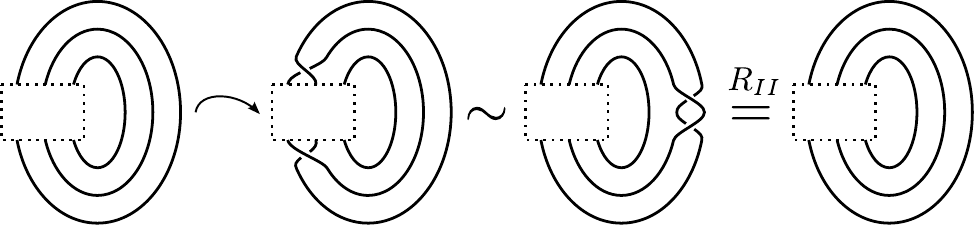}
\caption{Conjugation by $\sigma_i$ is trivial on the closure}
\label{fig_coniugio}
\end{figure}

\begin{definition}
Stabilization move:
\begin{align*}
B_m \ni c \mapsto c \sigma_m \in B_{m+1}
\end{align*}
where we can see $c$ in $B_{m+1}$ adding a string as in Figure \ref{fig_stab}.
\end{definition}

\begin{figure}[htb]
\centering
\includegraphics{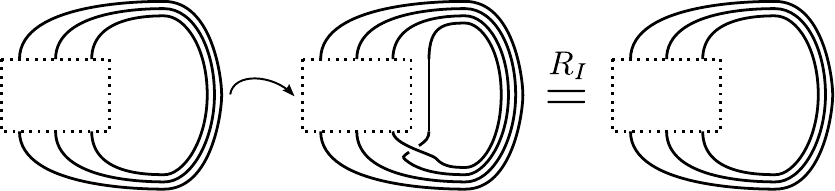}
\caption{Stabilization move is trivial on the closure}
\label{fig_stab}
\end{figure}

We can now state the main theorem:

\begin{theorem}[Markov]
Two braids whose closures are equivalent as links are related to each other by a sequence of conjugation and stabilization moves.
\end{theorem}

\chapter{Tools}

In order to prove Markov's Theorem, we need to introduce some other deformation moves. Let $L$ be a PL link, with an edge $[a,c]$ (where in general $[a_1,\dots,a_n]$ is the convex hull of the points $a_1,\dots,a_n$). We will write $ab$ to indicate the edge $[a,b]$ oriented from $a$ to $b$.

In this chapter, we suppose that $L$ is a link, $r$ is an axis such that the projection $P$ of $L$ on the plane $r^{\perp}$ is a link diagram and $L$ inherits the orientation from the axis $r$, so all of its edges has a well defined sign. Note that $L$ is in braid position if and only if all of its edges are positive.

\begin{definition}
The move that generates PL link isotopy, given by
\begin{align*}
L \mapsto L \setminus [a,c] \cup [a,b] \cup [b,c],
\end{align*}
is called $\mathcal{E}_{a,c}^b$. This move is licit if $L \cap [a,b,c] = [a,c]$.
\end{definition}

\begin{figure}[htb]
\centering
\includegraphics{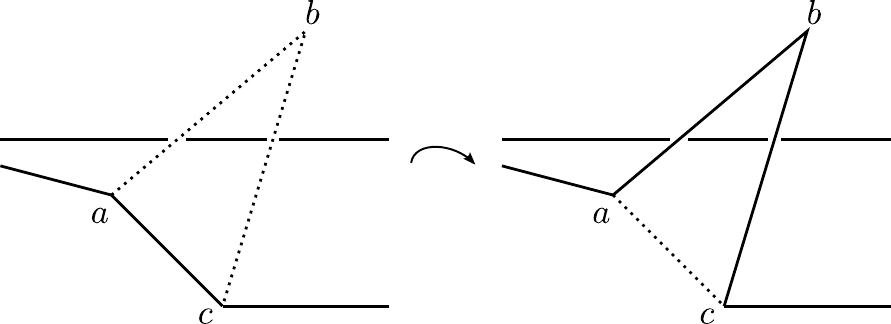}
\caption{The move $\mathcal{E}_{a,c}^b$}
\end{figure}

\begin{definition}
Given a negative edge $e=p_0 p_m$ of $L$, $m+1$ basepoints $p_0, p_1,\dots, p_m \in e$, such that $p_i,p_{i+1}$ is negative, and $m$ vertices $q_1, \dots, q_m \notin e$, the sawtooh move
\begin{align*}
\mathcal{S}_{p_0,\dots,p_m}^{q_1,\dots,q_m}= \mathcal{E}_{p_{m-1},p_m}^{q_m} \circ \dots \circ \mathcal{E}_{p_0,p_1}^{q_1}
\end{align*}
is the composition of type-$\mathcal{E}$ moves, each of which is licit when we operate it, and such that all $[p_i,q_{i+1}]$ and $[q_{i+1},p_{i+1}]$ are positive (Figure \ref{sawfigure}).
\end{definition}

\begin{figure}[htb]
\centering
\includegraphics{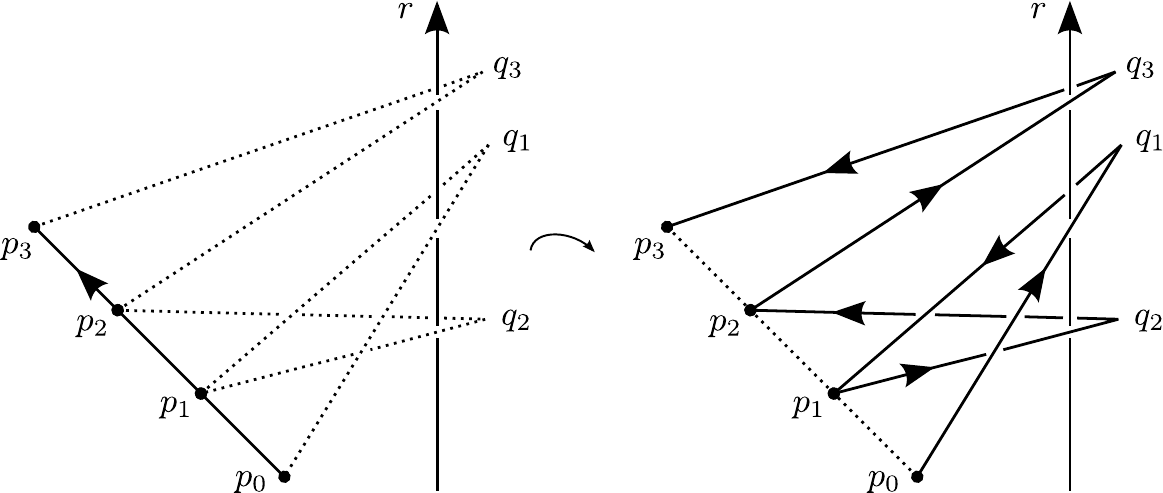}
\caption{The sawtooth move $\mathcal{S}_{p_0,\dots,p_3}^{q_1,\dots,q_3}$}
\label{sawfigure}
\end{figure}

\begin{lemma}[Existence of a Sawtooth]\label{sawexi}
Given a negative edge $e$ of $L$, there exists a sawtooth $\mathcal{S}$ on $e$. Moreover:
\begin{itemize}
\item[1.] if $\Omega$ is a convex set such that $\Omega \cap e$ is either empty or a vertex of $e$, then $\mathcal{S}$ can be chosen such as to avoid $\Omega$; formally, if $p_0,\dots,p_m$ are the basepoints and $q_1,\dots,q_m$ are the vertices of $\mathcal{S}$, then for every $i$
\begin{align*}
\Omega \cap [p_i,p_{i+1},q_{i+1}] \subseteq \Omega \cap e;
\end{align*}
\item[2.] if $e'$ is another edge of $L$ and $e \cap e'$ is one of the vertices of $e$, than $\mathcal{S}$ can be chosen such that (using the notation above) $[p_0,\dots,p_m,q_1,\dots,q_m] \cap e'$ is exactly that vertex.
\end{itemize}
\end{lemma}

\begin{proof}

We first consider the case where the link $L$ has a crossing point over the edge $e$. Let $c \in e$ and $d$ on another edge $f$ be the two points that project to the same point on $r^{\perp}$. We will assume that $c \in r^{\perp}$, for simplicity reasons. Let $\pi$ be the plane containing $c$ and $r$ and let $r'$ be the parallel to $r$ containing $c$ (Figure \ref{fig_saw1}).


\begin{figure}[htb]
\centering
\includegraphics{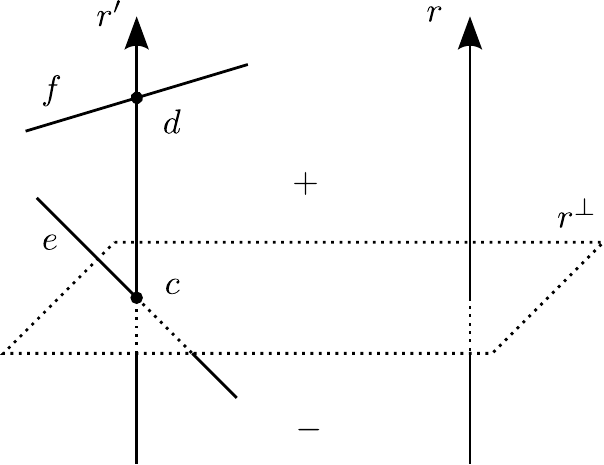}
\caption{Near a crossing point}
\label{fig_saw1}
\end{figure}

The link $L$ intersects $\pi$ in a finite number of points, because it is compact and in general position with respect to $r^{\perp}$. Hence $r' \cap L$ only contains $c$ and $d$, while $f \cap \pi=\{d\}$, otherwhise the rotation axis $r$ would not induce any sign on it. The plane $r^{\perp}$ divides the space into two semi-spaces (positive and negative), according to the orientation on $r$. So, in one of the two subspaces, we can draw a half-line lying on the plane $\pi$, starting from $c$, that does not intersect $L$ in any other point: if suffices that the angle it forms with $r'$ is small enough (and different from $0$). Then the half-line is almost parallel to the line $r'$. Let us take the points $\tilde{q}$ ad $q$ on this half-line, such that $\tilde{q}$ is the intersection between it and $r$, while $q$ does not belong to the line segment $c \tilde{q}$ (Figure \ref{fig_saw2}). 


\begin{figure}[htb]
\centering
\includegraphics{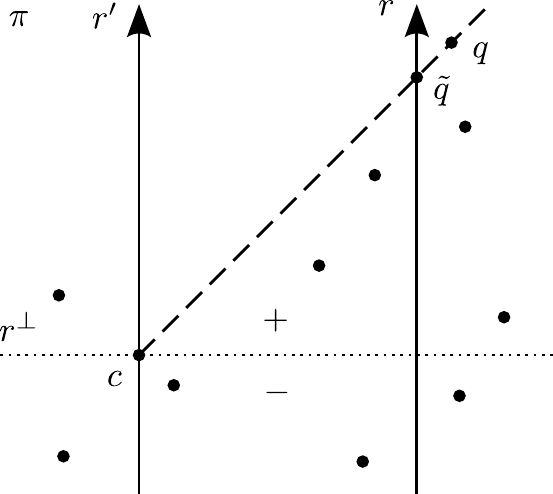}
\caption{Construction of $\tilde{q}$ and $q$}
\label{fig_saw2}
\end{figure}

Let $p$ and $p'$ be two points on $e$ such that $pp'$ is negative and $c$ is internal to it. We can choose $p$ and $p'$ close enough to $c$, and $q$ close enough to $\tilde{q}$, such that $[p,q,p'] \cap L = [p,p']$. So the move $\mathcal{E}_{p,p'}^q$ is licit and the edges $pq$ and $qp'$ are positive; it will be a tooth in our sawtooth (Figure \ref{fig_saw3}).


\begin{figure}[htb]
\centering
\includegraphics{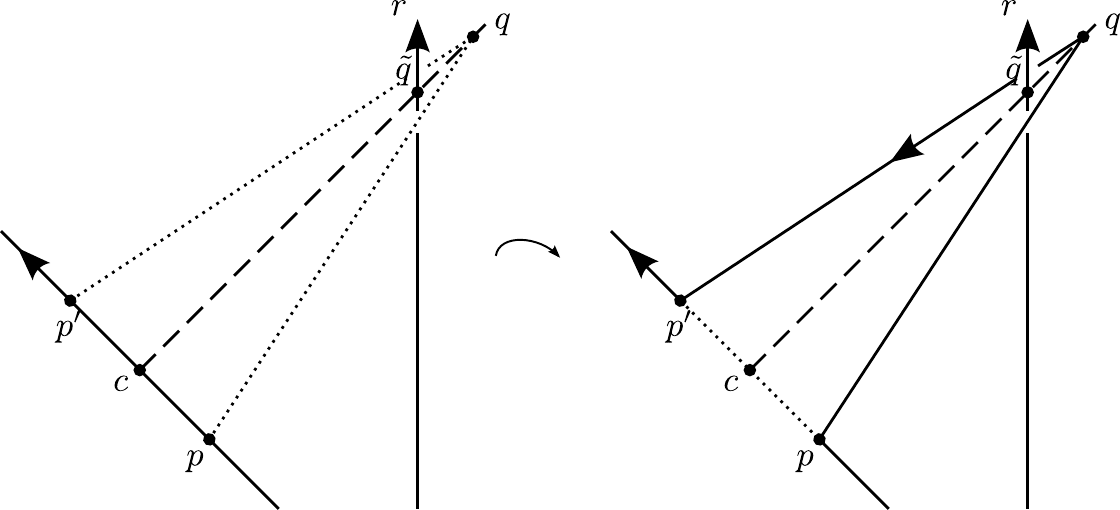}
\caption{Construction of a tooth}
\label{fig_saw3}
\end{figure}

We now consider the case where $e=[p_0,p_1]$ has no crossing points. Let $\Xi$ be the strip defined by the parallels to $r$ that intersect $e$. Then $\Xi \cap L$ only contains $e$, hence we can rotate the strip around $e$ by a sufficiently small angle $\phi$, so that the image $\Xi'$ of the strip $\Xi$ intersects $r$ in a point $\tilde{q_1}$, and still only intersects $L$ in $e$. Let us consider the triangle whose basis is $e$ and whose vertex is a point $q_1$ close to $\tilde{q_1}$, belonging to the same half-strip and such that $p_0 q_1$ and $q_1 p_1$ are positive line segments. So the move $\mathcal{E}_{p_0,p_1}^{q_1}$ is licit and is a sawtooth. (Figure \ref{saw5}).


\begin{figure}[htb]
\centering
\includegraphics{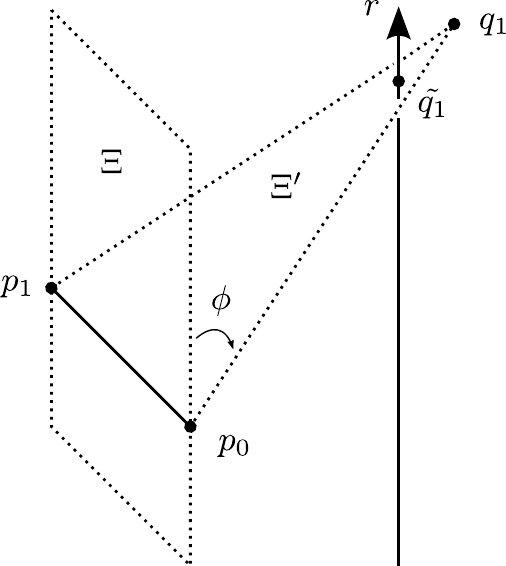}
\caption{No crossing points}
\label{saw5}
\end{figure}
\newpage
Last, we adapt the proof in order to show $1.$ and $2.$:


\begin{itemize}
\item[1.] If $\Omega$ is as stated, $e$ has a crossing point $c$ and $r'$ is as above, then $\Omega$ can intersect only one of the two half-lines that $c$ divides $r'$ into: if it did not, being a convex set it would intersect also $e$ in its internal point $c$, a contradiction. Then, if $\Omega$ does not intersect the positive (resp. negative) half-line, we can choose $c \tilde{q}$ to be in the positive (resp. negative) half-space and the proof above continues to work. Also, if there are not crossing points, it suffices to take one of the half-strips of $\Xi$ that does not intersect $\Omega$ (it is possible thanks to the convexity of $\Omega$) and rotate it as above in order to find an intersection point $\tilde{q_1}$ with $r$. If the rotation angle is small enough, the half-strip still does not intersect $\Omega$.
\item[2.]Let $\{a\}=e \cap e'$. We can still consider the strip $\Xi$ over the edge $e$ defined as above. Since the link is in general position with respect to $r^{\perp}$, $\Xi$ only intersects $e'$ in $a$. Now we can execute all the steps above taking all the vertices of the sawtooth in the same half-space (no matter which) and far from $r^{\perp}$ and $e'$, so that near $e'$ the convex hull of the sawtooth is indistinguishable from its projection on the strip $\Xi$. Then it only intersects $e'$ in $a$, too.

\end{itemize}
This concludes the proof.
\end{proof}

Some other useful deformation moves are the following.

\begin{definition}
The move $\mathcal{R}_{a,c}^b$ is given by $\mathcal{E}_{a,c}^b$ with the edges $ac$, $ab$ and $bc$ all positive.
\end{definition}

\begin{definition}
The move $\mathcal{W}_{a,d}^{b,c}$ is defined when $ab$, $bc$, $cd$, $ad$, $ca$ and $db$ are all positive and is given by the composition:
\begin{align*}
\mathcal{W}_{a,d}^{b,c} = \mathcal{E}_{b,d}^c \circ \mathcal{E}_{a,d}^b
\end{align*}
where the two type-$\mathcal{E}$ moves are licit.
\end{definition}

\begin{remark}
We can obtain the same link using another definition of $\mathcal{W}_{a,d}^{b,c}$, namely:
\begin{align*}
\mathcal{E}_{a,c}^b \circ \mathcal{E}_{a,d}^c
\end{align*}
\end{remark}

\begin{figure}[htb]
\centering
\includegraphics{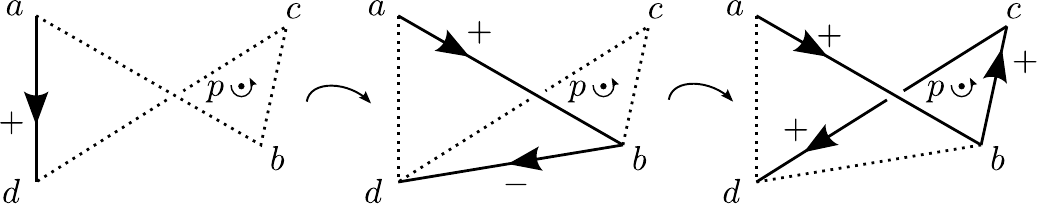}
\caption{Crossing type-$\mathcal{W}$ move}
\label{mosseWc}
\end{figure}

\begin{figure}[htb]
\centering
\includegraphics{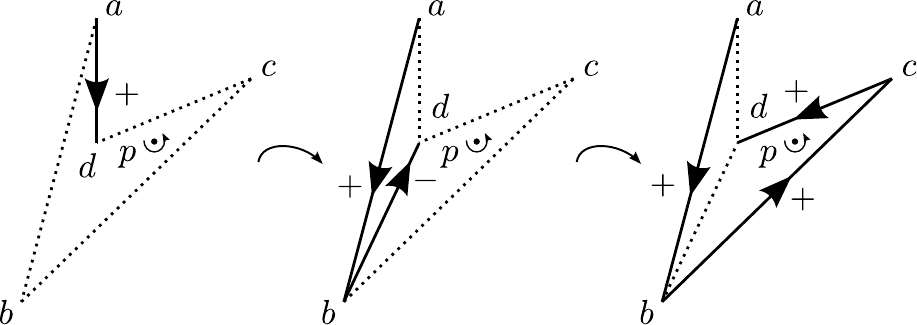}
\caption{Non-crossing type-$\mathcal{W}$ move}
\label{mosseWn}
\end{figure}

\begin{remark}
There are two possible type-$\mathcal{W}$ moves (Figures \ref{mosseWc} and \ref{mosseWn}). We can see this using the diagram obtained by projecting the link and all the involved line segments on $r^{\perp}$. Let $p$ be the intersection between $r$ and $r^{\perp}$. The point $b$ belongs to one of the two half-planes delimited by the line containing $a$ and $d$; if this is the half-plane containing $p$, then we get the crossing type-$\mathcal{W}$ move, otherwise we get the non-crossing type-$\mathcal{W}$ move.
\end{remark}

\begin{proposition}\label{propA}
Two braids whose closures are connected by a type-$\mathcal{R}$ move are conjugate to each other.
\end{proposition}

\begin{proof}
Let the $\mathcal{R}$ move be some $\mathcal{R}_{a,c}^b$.

First, we can notice that, up to subdivision, we can consider take the triangle $[a,b,c]$ to be arbitrarily small. Indeed, if $x \in [a,c]$, $y \in [a,b]$, $z \in [b,c]$ the following equality holds (Figure \ref{fig_suddivisione}):

\begin{align*}
\mathcal{R}_{a,c}^b = (\mathcal{R}_{b,c}^z)^{-1} \circ (\mathcal{R}_{a,b}^y)^{-1} \circ \mathcal{R}_{y,z}^b \circ (\mathcal{R}_{y,z}^x)^{-1} \circ \mathcal{R}_{x,c}^z \circ \mathcal{R}_{a,x}^y \circ \mathcal{R}_{a,c}^x.
\end{align*}

\begin{figure}[htb]
\centering
\includegraphics{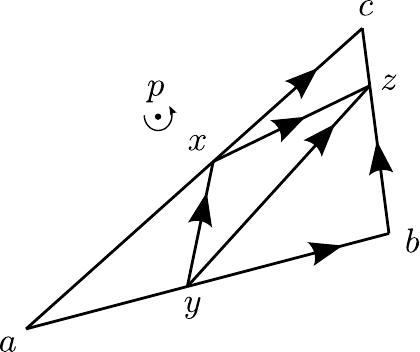}
\caption{Factorization of a type-$\mathcal{R}$ move}
\label{fig_suddivisione}
\end{figure}

So, after a finite number of subdivisions, each obtained by some type-$\mathcal{R}$ moves, we can separately work on smaller triangles.

From this moment on, we will only consider braid position links. Let us fix an angle $\theta_0$ and consequently a half-line $\theta=\theta_0$, which is the starting point of every braid position diagram.

Since the triangle $[a,b,c]$ involved in the type-$\mathcal{R}$ move is small, we may assume that only one of the following happens, explained in Figure \ref{staminchia}:

\begin{itemize}
\item the internal part of $[a,b,c]$ only contains one vertex of $L$;
\item the internal part of $[a,b,c]$ only contains one crossing of $L$;
\item only an edge of $L$ passes through the internal part of $[a,b,c]$, while the half-line $\theta=\theta_0$ does not;
\item an edge of $L$ and the half-line $\theta=\theta_0$ pass through the internal part of $[a,b,c]$.
\end{itemize}

\begin{figure}[htb]
\centering
\subfloat{\includegraphics{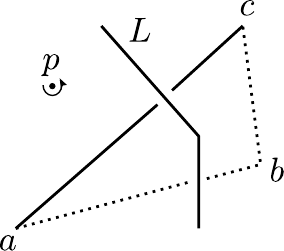}}\qquad
\subfloat{\includegraphics{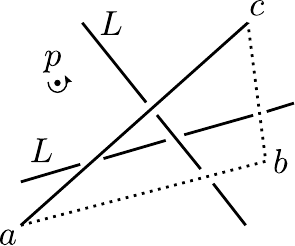}}\qquad
\subfloat{\includegraphics{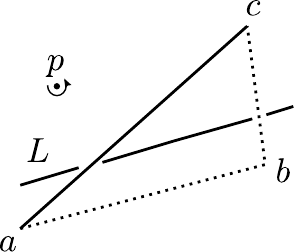}}\qquad
\subfloat{\includegraphics{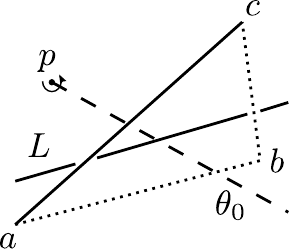}}
\caption{The $4$ cases for the type-$\mathcal{R}$}
\label{staminchia}
\end{figure}

In the first case, we simply have an equivalence of braids; to see that, it suffices to read the braids corresponding to $L$ and to $\mathcal{R}(L)$ (that is the link $L$ after the move $\mathcal{R}$), starting from $\theta=\theta_0$: the only difference between the two of them may be near the edge $[a,c]$ in $L$ and the edges $[a,b]$ and $[b,c]$ in $\mathcal{R}(L)$. In the diagram, the link intersects the boundary of the triangle $[a,b,c]$ twice, and it always passes over the triangle or under it. Reading the braid as an element of $B_n$, we find exactly the same generators, as in Figure \ref{propatreccia}.

\begin{figure}[htb]
\centering
\includegraphics{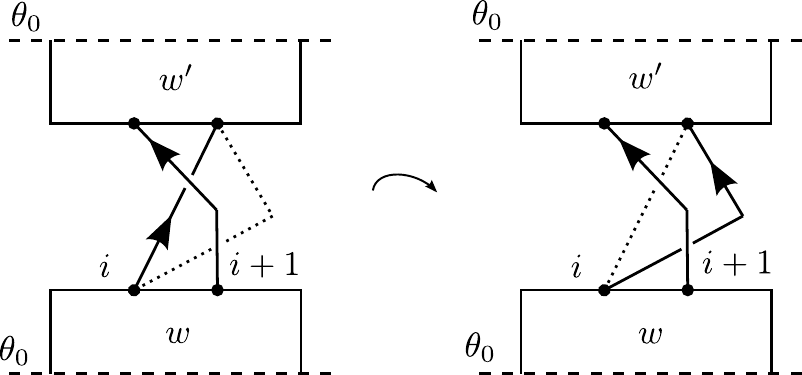}
\caption{The equivalence of $L$ and $\mathcal{R}(L)$}
\label{propatreccia}
\end{figure}

The second and third cases are similar: we have another equivalence of braids. Let us focus on the fourth case. We may assume that the half-line $\theta=\theta_0$ also crosses $L$ in the internal part of $[a,b,c]$; otherwise, up to subdivision, we can get the third case. We can notice that both the edge of $L$ and the half-line $\theta=\theta_0$ intersect twice the boundary of $[a,b,c]$, obtaining four different points (otherwise the edge of $L$ would have no sign) as showed in Figure \ref{hofame}.

\begin{figure}[htb]
\centering
\includegraphics{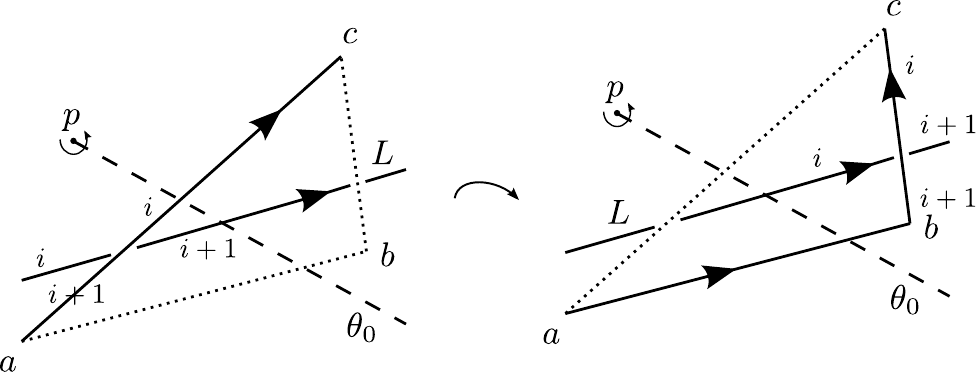}
\caption{Type-$\mathcal{R}$ move, fourth case}
\label{hofame}
\end{figure}

So, the braids we get before and after doing the move $\mathcal{R}$ can be (in any order): $w \sigma_i$ and $\sigma_i w$ or $w \sigma_i^{-1}$ and $\sigma_i^{-1} w$ (Figure \ref{tantafame}).

\begin{figure}[htb]
\centering
\includegraphics{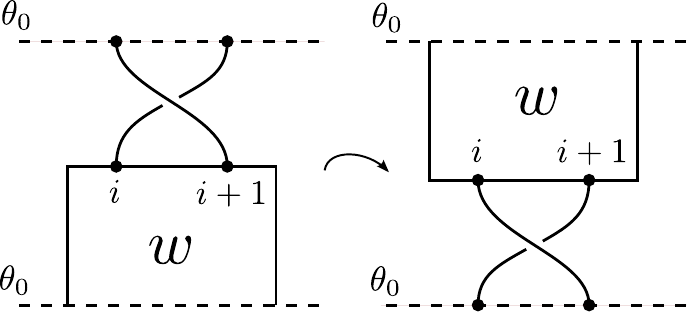}
\caption{Correspondent braid diagram}
\label{tantafame}
\end{figure}

In each case, they are conjugate in $B_n$.
\end{proof}

\begin{definition}
The move $\mathcal{W}_{a,d}^{b,c}$ is said to be empty if $[a,b,c,d] \cap L = [a,d]$.
\end{definition}

\begin{proposition}\label{propB}
A type-$\mathcal{W}$ move factors through two type-$\mathcal{R}$ moves and an empty type-$\mathcal{W}$ move.
\end{proposition}

\begin{proof}
Let us write $\mathcal{W}_{a,d}^{b,c} = \mathcal{E}_{d,b}^c \circ \mathcal{E}_{a,d}^b$, with $ad$, $bc$, $cd$, $ab$ and $db$ all positive. We could also write $\mathcal{W}_{a,d}^{b,c} = \mathcal{E}_{c,a}^b \circ \mathcal{E}_{a,d}^c$, but the proof would be the same, role-reversing $a$ and $d$, and $b$ and $c$. We will only consider the first case. By hypothesis:
\begin{align*}
L \cap [a,d,b]=[a,d]\\
L \cap [b,c,d]= \{ b \}.
\end{align*}

There exists a neighbourhood $U$ of $d$ in $[a,d]$ such that for every $x \in U$:
\begin{itemize}
\item $[x,d,b,c] \cap L=[x,d]$, since $L$ is compact;
\item $xb$ has the same sign as $db$.
\end{itemize}
Let us take $x$ in this neighbourhood. So the move factors as follows (Figure \ref{prop_b}):
\begin{align*}
\mathcal{W}_{a,d}^{b,c} = (\mathcal{E}_{a,b}^x)^{-1} \circ \mathcal{W}_{x,d}^{b,c} \circ \mathcal{E}_{a,d}^x.
\end{align*}

\begin{figure}[htb]
\centering
\includegraphics{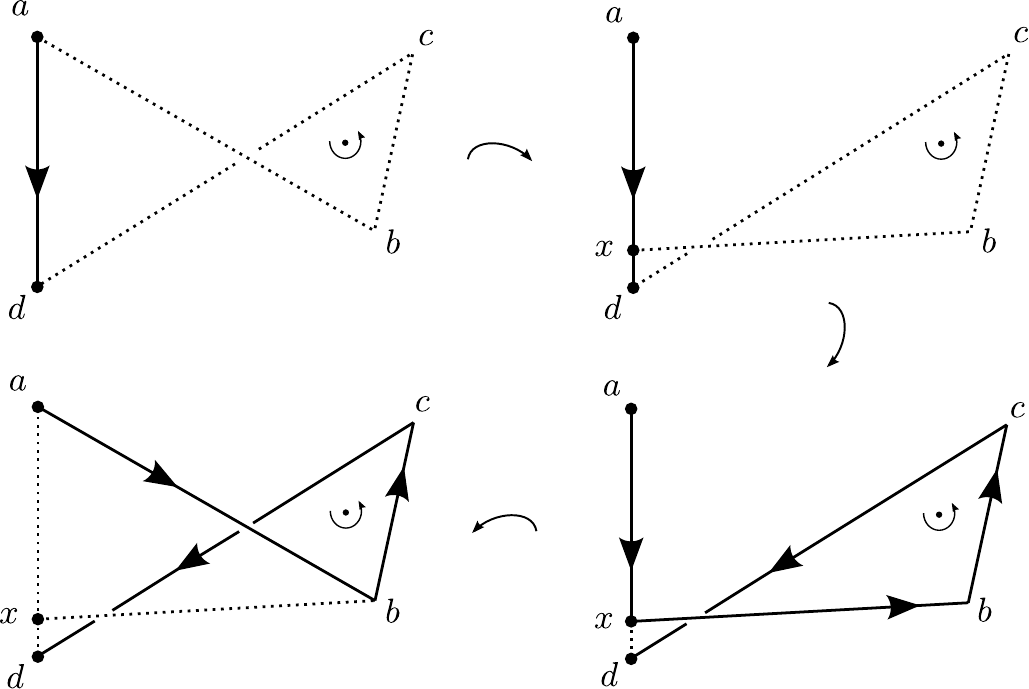}
\caption{Factorization of $\mathcal{W}_{a,d}^{b,c}$}
\label{prop_b}
\end{figure}

All the involved edges are positive:
\begin{itemize}
\item $ad$, $ab$, $bc$ and $cd$ by hypothesis;
\item $ax$ and $xd$ have the same sign as $ad$, since $x \in [a,d]$;
\item $xb$ has the same sign as $db$ by the choice of $x$.
\end{itemize}

So the two type-$\mathcal{E}$ moves are actually of type-$\mathcal{R}$ and the type-$\mathcal{W}$ move is well defined; moreover it is empty, since $[x,d,b,c] \cap L=[x,d]$.
\end{proof}

Recall that $P$ is the projection on $r^{\perp}$.
\begin{definition}
The move $\mathcal{W}_{a,d}^{b,c}$ is said to be really empty if $P([a,b,c,d]) \cap P(L)= P([a,d])$.
\end{definition}

\begin{proposition}\label{propC}
An empty type-$\mathcal{W}$ move can be decomposed into some type- $\mathcal{R}$ moves and one really empty crossing type-$\mathcal{W}$ move.
\end{proposition}

\begin{proof}
Let us suppose that $\mathcal{W}=\mathcal{W}_{a,d}^{b,c}$ is non-crossing. We will show that it can be decomposed into two type-$\mathcal{R}$ moves and an empty crossing type-$\mathcal{W}$ move (Figure \ref{fig_prop_c}). Let us consider the half-line originated from $p$ that contains $d$, then let us rotate it by a small negative angle (with respect to the orientation given by $r$) on $r^{\perp}$. Let $x$ be a point on it, in the half-plane delimited by $[a,b]$ not containing $d$, close to $[a,b]$, so that the line segments $ax$, $xb$, $cx$ and $xd$ are all positive. Hence the move $\mathcal{W}_{a,x}^{b,c}$, operated after $\mathcal{R}_{a,d}^x$ is licit, non-crossing and:
\begin{align*}
\mathcal{W}_{a,d}^{b,c}=(\mathcal{R}_{c,d}^x)^{-1} \circ \mathcal{W}_{a,x}^{b,c} \circ \mathcal{R}_{a,d}^x.
\end{align*}

Moreover, the move $\mathcal{W}_{a,x}^{b,c}$ is still empty, since the convex hull $[a,b,c,x]$ is almost contained in $[a,b,c,d]$.

So we may focus on the case $\mathcal{W}$ is crossing. We can apply a homothety of ratio $t>1$, with respect to a center $q \in r \cap \inter[a,b,c,d]$, to $M=L \setminus [a,d]$. If we choose $t$ large enough, then the image $M_t$ of $M$ will be disjoint from $[a,b,c,d]$ and its projection will be disjoint from $P([a,b,c,d])$. Let $\tilde{a}$ and $\tilde{d}$ be a perturbation of the images of $a$ and $d$ respectively, such that $\tilde{a}a$ and $d \tilde{d}$ are both positive.

Then:
\begin{align*}
L_t&= [\tilde{a},a] \cup [a,d] \cup [d,\tilde{d}] \cup M_t \text{ and}\\
L'_t&= [\tilde{a},a] \cup [a,b] \cup [b,c] \cup [c,d] \cup [d,\tilde{d}] \cup M_t
\end{align*}
are joined by the really empty crossing move $\mathcal{W}_{a,d}^{b,c}$ Figure \ref{omotetia}). Finally, $L \rightarrow L_t$ and $L'_t \rightarrow L'$ are small perturbations of a homothety, hence factor through a sequence of type-$\mathcal{R}$ moves.
\end{proof}

By Proposition \ref{propA}, type-$\mathcal{R}$ moves are conjugation moves, and by Proposition \ref{propB} and Proposition \ref{propC} a type-$\mathcal{W}$ move can be decomposed into other conjugation moves and a really empty type-$\mathcal{W}$ move, that is a stabilization. In order to prove Markov's Theorem it is enough to show the following:

\begin{theorem}
Two braids whose closures are equivalent as links, are related to each other by a sequence of type-$\mathcal{R}$ and $\mathcal{W}$ moves.
\end{theorem}

Another useful deformation move is the following:

\begin{definition}
Given $a$, $b$, $c$, $d$ such that $ab$, $ac$, $bd$, $cd$ are all positive, while $ad$ is negative, we set (Figure \ref{typeT})
\begin{align*}
\mathcal{T}_{a,b,d}^{a,c,d} = \mathcal{E}_{a,d}^c \circ (\mathcal{E}_{a,d}^b)^{-1}.
\end{align*}
\end{definition}

\begin{proposition}\label{dect}
A type-$\mathcal{T}$ move can be decomposed into type-$\mathcal{R}$ and $\mathcal{W}$ moves.
\end{proposition}

\begin{proof}
Let us take a point $a' \in [a,b,d]$ close to $a$ and very close to $[a,b]$, and $d' \in [a,c,d]$ close to $d$ and very close to $[c,d]$. So the line segments $aa'$ and $a'b$ are positive, since $ab$ is, and $cd'$, $d'd$ are also positive, since $cd$ is. Consequently, the moves $\mathcal{E}_{a,b}^{a'}$ and $\mathcal{E}_{c,d}^{d'}$ are type-$\mathcal{R}$ moves.

Now, the move $\mathcal{W}_{a,a'}^{c,d'}$ is well defined, because $aa'$ is positive as noticed above, $ac$ and $cd$ are positive by hypothesis, $d'a'$ has the same sign as $da$, so it is positive too, and finally both the auxiliary edges $a'c$ and $d'a$ are positive since $ac$ and $da$ are. Similarly, $\mathcal{W}_{d',d}^{a',b}$ is well defined.

So the following holds (Figure \ref{fig_dec_t})
\begin{align*}
\mathcal{T}_{a,b,d}^{a,c,d} = (\mathcal{R}_{c,d}^{d'})^{-1} \circ (\mathcal{W}_{d,d'}^{a',b})^{-1} \circ \mathcal{W}_{a,a'}^{c,d'} \circ \mathcal{R}_{a,b}^{a'}
\end{align*}
which concludes the proof.
\end{proof}

\begin{proposition} \label{cambiosaw}
Two sawteeth $\mathcal{S}$ and $\mathcal{S}'$ erected on the same negative edge $e$ are connected by a chain of type-$\mathcal{R}$ and $\mathcal{W}$ moves.
\end{proposition}

\begin{proof}
We first claim that given a point $d \in e$ and a sawtooth $\mathcal{S}=\mathcal{S}_{a_0,\dots,a_m}^{b_1,\dots,b_m}$ on $e$, we can construct a finer sawtooth with $d$ among its basepoints. If $d$ is already among the basepoints of $\mathcal{S}$, there is nothing to show. Otherwise, let us suppose that $d$ belongs to the internal part of $[a_i,a_{i+1}]$, where $a_i$ and $a_{i+1}$ are two basepoints of $\mathcal{S}$. We can assume, up to moving the vertex $b_{i+1}$ using a licit type-$\mathcal{T}$ move, that the edge $[b_{i+1},d]$ is positive or negative. Let us consider the first case (the other is similar). Let $c$ be a point close to the triangle $[a_i,a_{i+1},b_{i+1}]$ and such that $dc$ and $ca_{i+1}$ are positive. Hence the move $\mathcal{W}_{b_{i+1},a_{i+1}}^{d,c}$ is licit and produces a finer sawtooth with $d$ among its basepoints, as shown in Figure \ref{fig_cambio_saw}.

Getting to the statement, by the claim we may assume that $\mathcal{S}$ and $\mathcal{S}'$ have the same basepoints. We operate a type-$\mathcal{T}$ move to each of the type-$\mathcal{E}$ moves that define $\mathcal{S}$, thus obtaining the sawtooth $\mathcal{S}'$. This is permitted, since after every operation of type $\mathcal{T}$, we get another intermediate sawtooth. Finally, thanks to Proposition \ref{dect}, the type-$\mathcal{T}$ moves can be decomposed into type-$\mathcal{R}$ and $\mathcal{W}$ moves.
\end{proof}
\newpage
\begin{figure}[p]
\centering
\includegraphics{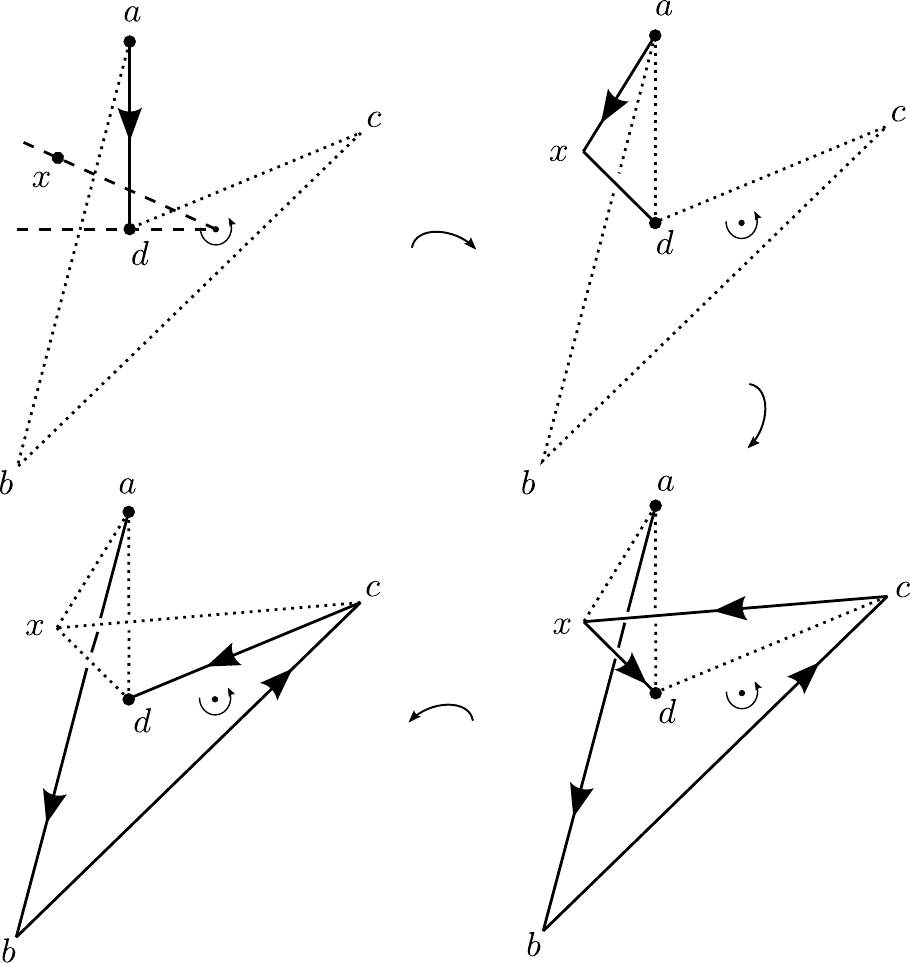}
\caption{The decomposition of a non-crossing type-$\mathcal{W}$ move}
\label{fig_prop_c}
\end{figure}
\begin{figure}[p]
\centering
\includegraphics{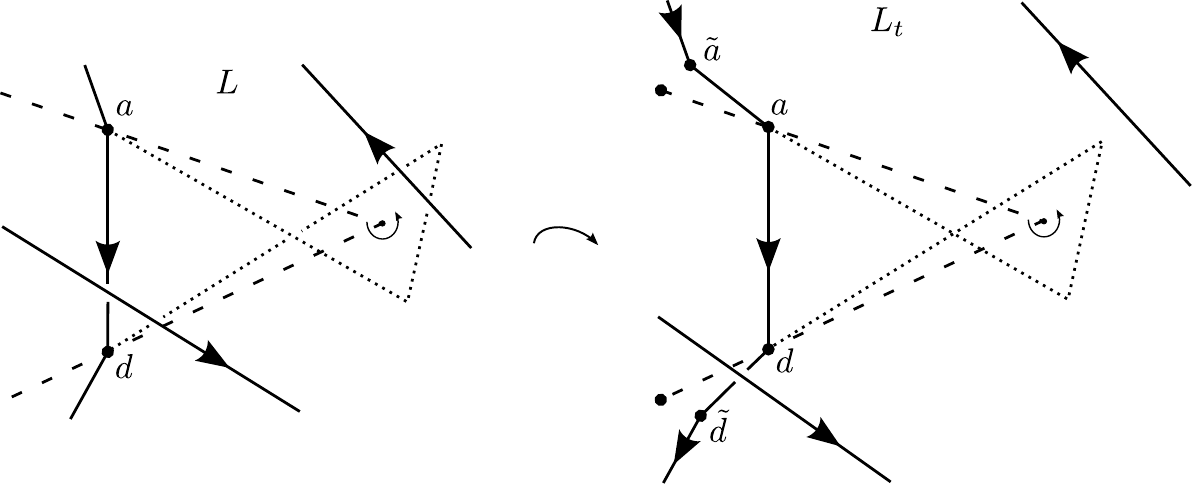}
\caption{The decomposition in a really empty type-$\mathcal{W}$ move and $\mathcal{R}$ moves}
\label{omotetia}
\end{figure}
\begin{figure}[p]
\centering
\includegraphics{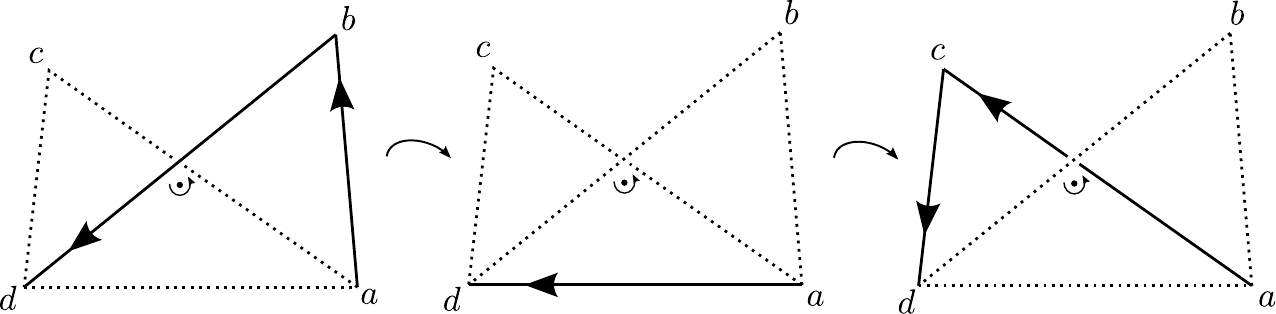}
\caption{A type-$\mathcal{T}$ move}
\label{typeT}
\end{figure}

\begin{figure}[p]
\centering
\includegraphics{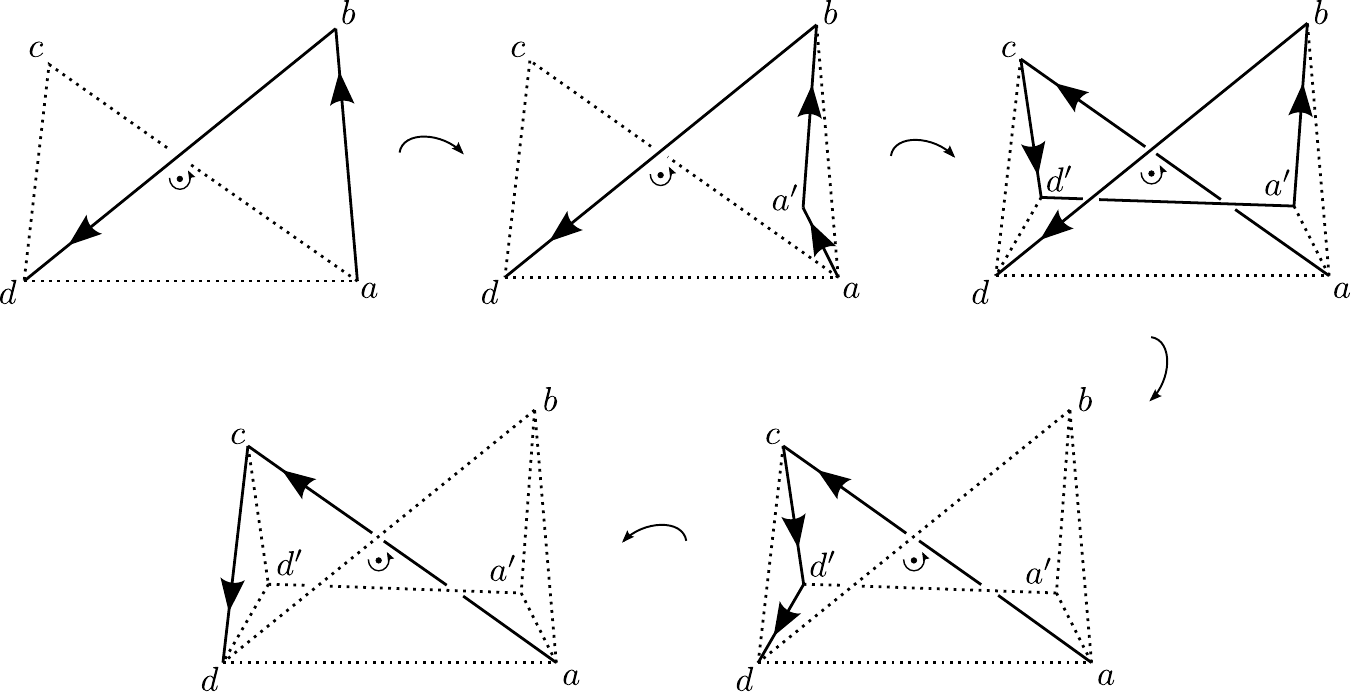}
\caption{Decomposition of a type-$\mathcal{T}$ move}
\label{fig_dec_t}
\end{figure}

\begin{figure}[p]
\centering
\includegraphics{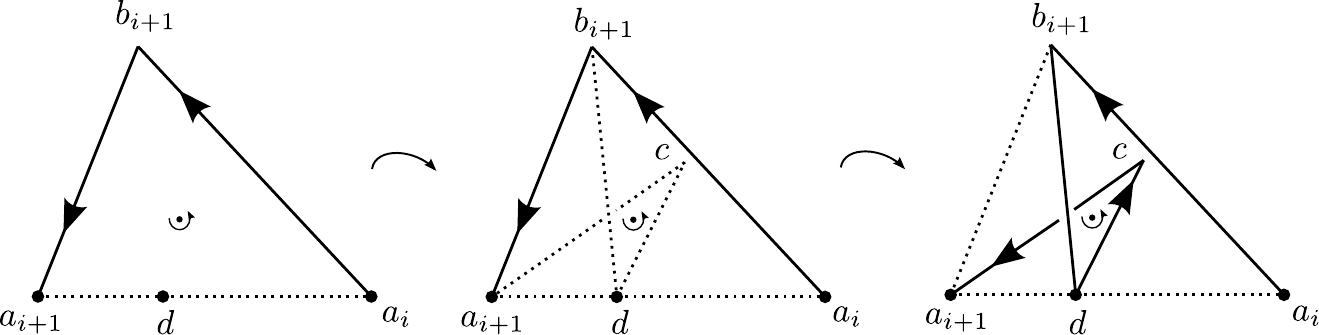}
\caption{Including $d$ in a sawtooth}
\label{fig_cambio_saw}
\end{figure}

\chapter{Markov's Theorem: the proof}

Given a $PL$ link $L$ in general position with respect to the rotation axis $r$, so that every edge has a sign, we can define the height $h(L)$ of $L$ as the number of negative edges of $L$. The link is in braid position if and only if its height is $0$, if and only if it is the closure of a braid. In general, two isotopic $PL$ links are joined by a chain of type-$\mathcal{E}$, $\mathcal{R}$, $\mathcal{S}$, $\mathcal{W}$ or $\mathcal{T}$ moves, so each operation may change the height.

\begin{remark} \label{h0}
If two $PL$ links $L$ and $L'$ are the closure of two braids and are joined by a chain of deformation moves of the above mentioned types, and if the height of each link of the chain is $0$, then the deformation moves which occur in the chain can only be of type-$\mathcal{R}$ and $\mathcal{W}$, therefore in order to prove Markov's Theorem, we only need to show that we can find a chain of deformation moves from $L$ to $L'$ that preserve the height.
\end{remark}

This will easily follow from the next two lemmas, to the proof of which this chapter is devoted.

\begin{lemma}\label{lemma1}
If two links $L$ and $L'$ are joined by a single deformation move of type-$\mathcal{E}$, $\mathcal{R}$, $\mathcal{S}$, $\mathcal{W}$ or $\mathcal{T}$ and $h(L)=h(L') >0$, then there is a chain of links
\begin{align*}
L=L_0 \rightarrow L_1 \rightarrow \dots \rightarrow L_n=L'
\end{align*}
each obtained from the previous one using a deformation move of the above mentioned type, such that $h(L_i)< h(L)$ for every $i=1,\dots,n-1$.
\end{lemma}

\begin{proof}
The move joining $L$ and $L'$ can only be of type $\mathcal{R}$, $\mathcal{W}$ or $\mathcal{E}$ sending a negative edge to a positive and a negative one.

Since $h(L)>0$, then $L$ has a negative edge, namely $xy$. If the move is of type $\mathcal{R}$ or $\mathcal{W}$, then it does not apply to $xy$, because it only involves positive edges. In each case we can erect a sawtooth on $xy$ that avoids the convex hull of the points involved in the move; formally, if the sawtooth is $\mathcal{S}_{p_0,\dots,p_n}^{q_1,\dots,q_n}$:
\begin{itemize}
\item if $\mathcal{R}=\mathcal{R}_{a,b}^c$, then for each $i$: $[a,b,c] \cap [p_i,p_{i+1},q_{i+1}] = [p_i,p_{i+1}]$;
\item if $\mathcal{W}=\mathcal{W}_{a,d}^{b,c}$, then for each $i$: $[a,b,c,d] \cap [p_i,p_{i+1},q_{i+1}] = [p_i,p_{i+1}]$: this is possible because by the previous results we can assume that the move is of empty type.
\end{itemize}

If $h(L) > 1$, and the move is of type $\mathcal{E}$, we can erect a sawtooth on an edge which is not involved in the type-$\mathcal{E}$ move and, as above:
\begin{itemize}
\item if $\mathcal{E}=\mathcal{E}_{a,b}^c$, then for each $i$: $[a,b,c] \cap [p_i,p_{i+1},q_{i+1}] = [p_i,p_{i+1}]$.
\end{itemize}

Now let us suppose that $h(L)=1$ and fix $\mathcal{E}= \mathcal{E}_{a,c}^b$. Without loss of generality, $ac$ and $bc$ are negative, while $ab$ is positive. First, we suppose that $a$ and $b$ are close enough so that it is possible to erect a sawtooth on both $ac$ and $bc$ with the same vertices (and the same number of basepoints), all the involved edges have a sign and all the involved pyramids intersect the link only in the edge contained in $[a,c]$ and do not intersect with each other (Figure \ref{piramidi}).

\begin{figure}[htb]
\centering
\includegraphics{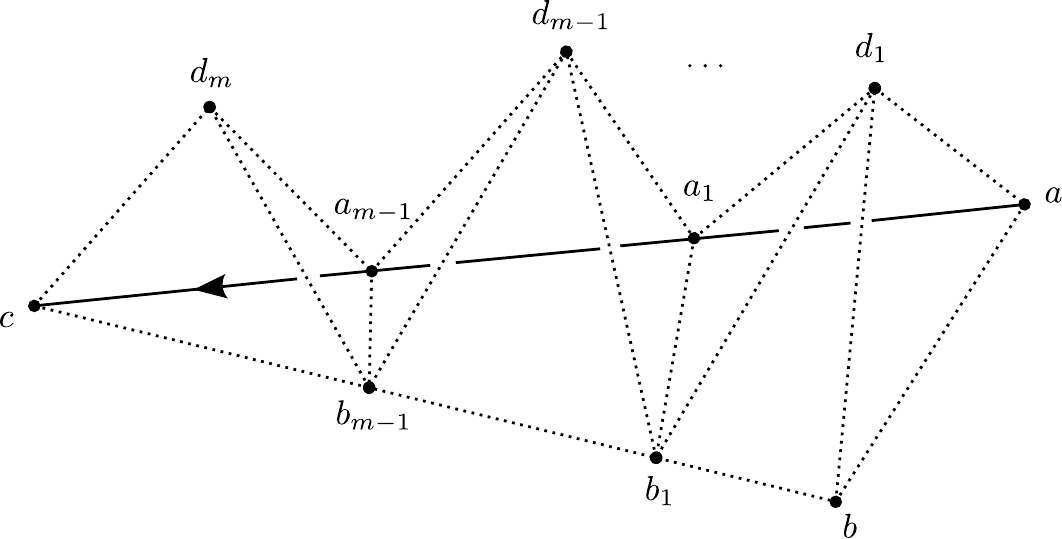}
\caption{Factorization of $\mathcal{E}_{a,c}^b$}
\label{piramidi}
\end{figure}

Let $\mathcal{S}_{\underline{a}}^{\underline{d}} =  \mathcal{S}_{a_0,\dots,a_m}^{d_1,\dots,d_m}$ and similarly $\mathcal{S}_{\underline{b}}^{\underline{d}}$ the sawteeth on $ac$ and $bc$ respectively. We will show that $\mathcal{E}_{a,c}^b$ can be decomposed as:
\begin{align} \label{schifo}
\mathcal{E}_{a,c}^b=(\mathcal{S}_{\underline{b}}^{\underline{d}})^{-1} \circ \mathcal{F}_0 \circ \mathcal{F}_1 \circ \dots \circ \mathcal{F}_{m-1} \circ \mathcal{S}_{\underline{a}}^{\underline{d}}
\end{align}
where the $\mathcal{F}_i$s are defined as follows:
\begin{itemize}
\item If $a_{m-1}b_{m-1}$ is positive, then the moves $\mathcal{R}_{a_{m-1},d_m}^{b_{m-1}}$ and $(\mathcal{R}_{d_{m-1},b_{m-1}}^{a_{m-1}})^{-1}$ are licit on the edge $a_{m-1} b_{m-1}$ after applying $\mathcal{S}_{\underline{a}}^{\underline{d}}$ and we define
\begin{align*}
\mathcal{F}_{m-1}=(\mathcal{R}_{d_{m-1},b_{m-1}}^{a_{m-1}})^{-1} \circ\mathcal{R}_{a_{m-1},d_m}^{b_{m-1}};
\end{align*}
\item If $a_{m-1}b_{m-1}$ is negative, then the moves $\mathcal{R}_{d_{m-1},a_{m-1}}^{b_{m-1}}$ and $(\mathcal{R}_{b_{m-1},d_m}^{a_{m-1}})^{-1}$ are licit and we define
\begin{align*}
\mathcal{F}_{m-1}=(\mathcal{R}_{b_{m-1},d_m}^{a_{m-1}})^{-1} \circ \mathcal{R}_{d_{m-1},a_{m-1}}^{b_{m-1}}.
\end{align*}
\end{itemize}

The case where $a_{m-1}b_{m-1}$ is positive is shown in Figure \ref{tantepiramidi}.

\begin{figure}[p]
\centering
\includegraphics{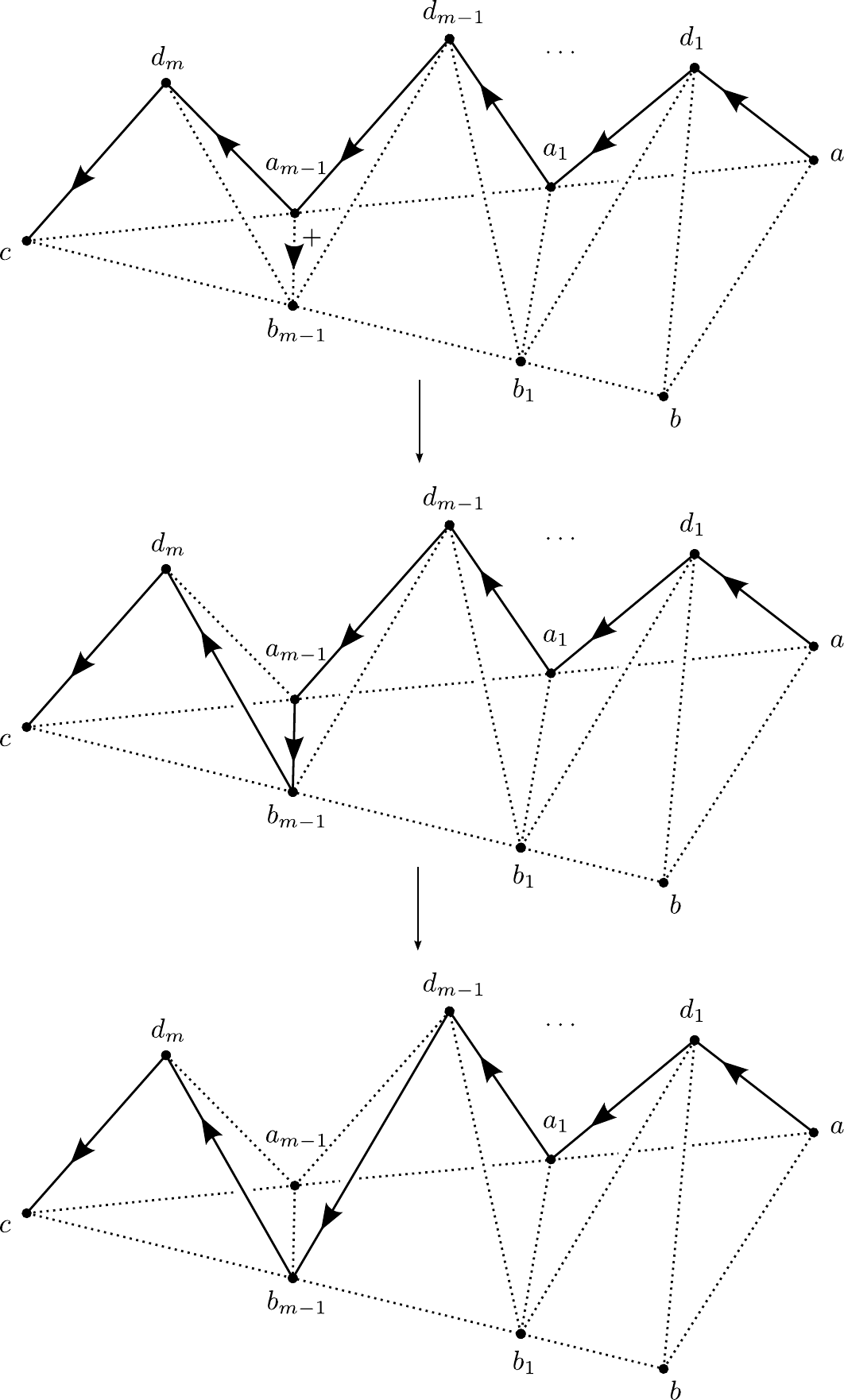}
\caption{The move $\mathcal{F}_{m-1}$, case $a_{m-1}b_{m-1}$ positive}
\label{tantepiramidi}
\end{figure}

In the same way we define the operations $\mathcal{F}_{m-2}, \dots, \mathcal{F}_1$. Finally we need to define $\mathcal{F}_0$. Since $ab$ is positive, we simply take $\mathcal{F}_0= \mathcal{R}_{a,d_1}^b$.

Let us study the chain \eqref{schifo}. The first move is a sawtooth, and decreases height by $1$. The type-$\mathcal{F}$ moves are compositions of type-$\mathcal{R}$ moves, and preserve height. The final $\mathcal{S}^{-1}$ operation increases the height by $1$ again.

In general, the existence of a sawtooth on a negative edge $xy$ given a set of vertices $\{q_1,\dots,q_m\}$ is guaranteed if there exist some points $p_0=x$, $p_1, \dots, p_m=y$ on $xy$ such that $p_i q_{i+1}$ and $q_{i+1} p_{i+1}$ are all positive and $[p_i,p_{i+1},q_{i+1}] \cap L = [p_i, p_{i+1}]$ for every $i$. This is an open condition on $y$. Moreover, if $xy'$ is another negative edge, and $y'$ is close to $y$, then we can erect a sawtooth on $xy'$ with the same vertices $\{q_1,\dots,q_m\}$ and basepoints $\{p'_0,\dots,p'_m\}$ such that $[p_i,p'_i,q_{i+1},p_{i+1},p'_{i+1}] \cap L = [p_i,p_{i+1}]$; this is an open condition on $y'$.

Now, given $z \in [a,b]$ and fixed a random sawtooth $\mathcal{S}$ on $zc$, let $U_{z}$ be an open connected neighbourhood containing only points $y \in [a,b]$ such that on $yc$ can be erected a sawtooth with the same vertices as $\mathcal{S}$ and all the pyramids avoid the link in the same sense as above. Hence $\bigcup_{z \in [a,b]} U_z = [a,b]$ and since $[a,b]$ is compact, there exists a finite subcover, given by: $z_0=a, z_1,\dots,z_n = b$. We can assume that the points $z_i$ are ordered on the edge $ab$. Since each $U_{z_i}$ is an open line segment, then $U_{z_i} \cap U_{z_{i+1}}$ is an open line segment too; let $x_i \in U_{z_i} \cap U_{z_{i+1}}$ for each $i$ (Figure \ref{ricopr}).

\begin{figure}[htb]
\centering
\includegraphics{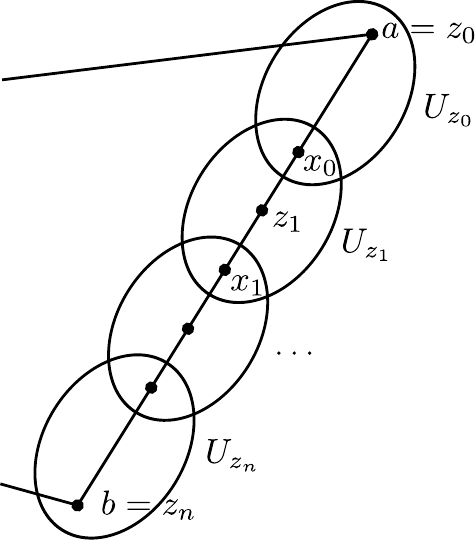}
\caption{The finite open cover $U_{z_0}, \dots, U_{z_n}$}
\label{ricopr}
\end{figure}

Now we can erect a sawtooth with common vertices on $z_ic$ and $x_ic$ and also on $x_ic$ and $z_{i+1}c$ for every $i$, and write the move $\mathcal{E}_{a,c}^b$ as a composition of moves of type \eqref{schifo}. Finally we change every $\mathcal{S} \circ \mathcal{S'}^{-1}$ move that appears using Proposition \ref{cambiosaw}, thus obtaining an initial move of type $\mathcal{S}^{-1}$ followed by the composition of type-$\mathcal{R}$ and $\mathcal{W}$ moves and finally another sawtooth. That means that with the first operation we decrease height, then we keep it constant until the last move, which increases it again.
\end{proof}

\begin{lemma}\label{lemma2}
If two links $L$ and $L'$ are joined by a chain of deformation moves
\begin{align*}
L \rightarrow L'' \rightarrow L'
\end{align*}
such that $h(L'') > h(L)$ and $h(L'') > h(L')$, then there is a chain of links
\begin{align*}
L=L_0 \rightarrow L_1 \rightarrow \dots \rightarrow L_n=L'
\end{align*}
such that $h(L_i) < h(L'')$ for each $i$.
\end{lemma}

\begin{proof}
The moves which increase height are the following:
\begin{itemize}
\item type-$\mathcal{S}^{-1}$ moves, that create a negative edge;
\item some of the type-$\mathcal{E}$ moves (Figure \ref{3casi}), precisely, if $\mathcal{E}=\mathcal{E}_{a,c}^b$:
\begin{itemize}
\item[1.] the move sends the positive edge $ac$ to a positive and a negative one (without loss of generality, we assume $bc$ is the negative edge);
\item[2.] the move sends the positive edge $ac$ to two negative edges;
\item[3.] the move sends the negative edge $ac$ to two negative edges.
\end{itemize}
\end{itemize}

\begin{figure}[htb]
\centering
\includegraphics{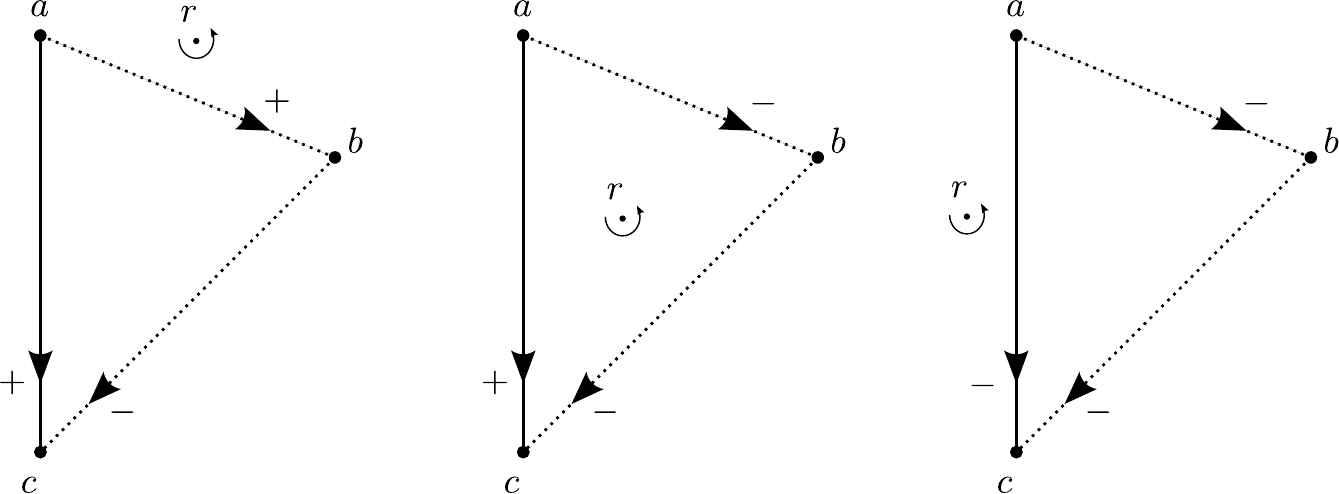}
\caption{The three possible type-$\mathcal{E}$ moves.}
\label{3casi}
\end{figure}

There is also another type-$\mathcal{E}$ move that increases height, namely $(\mathcal{E}_{a,c}^b)^{-1}$, where $ac$ is negative while $ab$ and $bc$ are both positive. Anyway this is a particular case of sawtooth, so we do not need to separately consider it.

The possible chains are the following:
\begin{itemize}
\item[(a)] first we apply a type-$\mathcal{S}^{-1}$ move, and then a $\mathcal{S}$ move;
\item[(b)] first we apply a type-$\mathcal{E}$ move, and then a $\mathcal{S}$ move (or viceversa first we apply a type-$\mathcal{S}^{-1}$ move, and then a $\mathcal{E}^{-1}$ move);
\item[(c)] first we apply a type-$\mathcal{E}$ move and then a $\mathcal{E}^{-1}$ move.
\end{itemize}

It is easy to see that the proof of the lemma in case (c) follows by case (b). Indeed if the first move is $\mathcal{E}_{a,c}^b$, where $bc$ is negative, and $(\mathcal{E}_{x,z}^y)^{-1}$ is the second one, then thanks to case $2$ of Lemma \ref{sawexi} we can erect a sawtooth $\mathcal{S}$ on $bc$ which avoids $ac$, and so
\begin{align*}
\mathcal{E}_{a,c}^b \circ (\mathcal{E}_{x,z}^y)^{-1}=\big(\mathcal{E}_{a,c}^b \circ \mathcal{S} \big) \circ \big( \mathcal{S}^{-1} \circ (\mathcal{E}_{x,z}^y)^{-1} \big).
\end{align*} 

Using case (b) separately on $\mathcal{E}_{a,c}^b \circ \mathcal{S} $ and $\mathcal{S}^{-1} \circ (\mathcal{E}_{x,z}^y)^{-1}$ we can prove case (c).

In case (a), let $\mathcal{S}_1^{-1}=(\mathcal{S}_{p_0,\dots,p_m}^{q_1,\dots,q_m})^{-1}: L \rightarrow L''$ be the first sawtooth and $\mathcal{S}_2=\mathcal{S}_{r_0,\dots,r_k}^{t_1,\dots,t_k}: L'' \rightarrow L'$ be the second one. If they are erected on the same edge, then by Lemma \ref{cambiosaw} we can factor $\mathcal{S}_2 \circ \mathcal{S}_1^{-1}$ using type-$\mathcal{R}$ and $\mathcal{W}$ moves, thus obtaining a chain of links with the same height. Now let us suppose that the two sawteeth are constructed on two different edges, namely $e$ and $e'$. Then the chain is as in Figure \ref{solosaw}.

\begin{figure}[htb]
\centering
\includegraphics{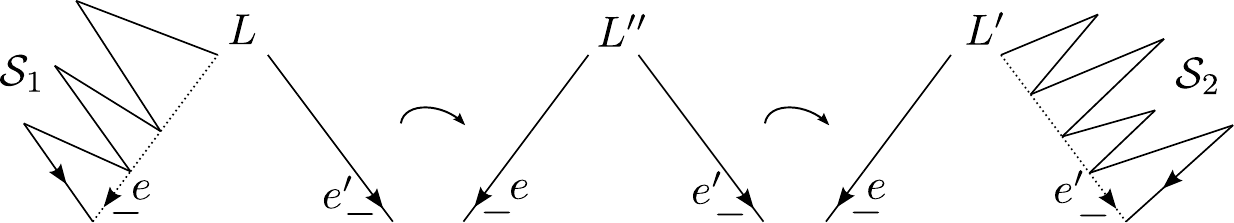}
\caption{Two sawteeth of different edges}
\label{solosaw}
\end{figure}

Now we will decompose this chain without reaching the level $h(L'')$ of height, as shown in Figure \ref{labelliamola}. Thanks to Lemma \ref{sawexi}, there exists a sawtooth on $e'$ that avoids the first tooth of $\mathcal{S}_1$, namely $[p_0,p_1,q_1]$. By applying it, we decrease height by $1$ and now we can apply the move $(\mathcal{E}_{p_0,p_1}^{q_1})^{-1}$. There exists another sawtooth on $e'$ which avoids the second tooth of $\mathcal{S}_1$. We can change the sawtooth with type-$\mathcal{R}$ and $\mathcal{W}$ moves, by Lemma \ref{cambiosaw}, preserving height. Now applying Lemma \ref{sawexi} again we erect a small sawtooth on $[p_0,p_1]$ in order to decrease height again, and apply a $\mathcal{S}^{-1}$ move destroying the whole sawtooth on $[p_0,p_2]$, thus obtaining a single negative edge. This move is licit, since the small sawtooth on $[p_0,p_1]$ has been erected after the one erected on $e'$, so this latter avoids every tooth of the small sawtooth and the triangle $[p_1,p_2,q_2]$ by construction. Iterating this process, we obtain the negative edge $e$ and a sawtooth on $e'$. Changing the sawtooth again with type-$\mathcal{R}$ and $\mathcal{W}$ moves, we get the link $L'$. As we have already noticed, the height is lower than $h(L'')$ during the whole process.

For case (b), it suffices to consider the chain
\begin{align*}
L \xrightarrow{\mathcal{E}} L'' \xrightarrow{\mathcal{S}} L'
\end{align*}
the other being identical. Fix $\mathcal{E}=\mathcal{E}_{a,c}^b$ and $\mathcal{S}=\mathcal{S}_{p_0,\dots,p_m}^{q_1,\dots,q_m}$. If the sawtooth is erected on an edge of $L''$ that also belongs to $L$ (that means not $ab$ nor $bc$), there exists another sawtooth $\mathcal{S}'$ on the same edge, avoiding the triangle $[a,b,c]$. So we can apply the following chain of deformation moves to $L$:
\begin{align*}
(\mathcal{S} \circ \mathcal{S}'^{-1}) \circ \mathcal{E} \circ \mathcal{S}'
\end{align*}
obtaining $L'$. Since $\mathcal{S} \circ \mathcal{S}'^{-1}$ can be decomposed into type-$\mathcal{R}$ and $\mathcal{W}$ moves, the height is lower than $h(L'')$ during the whole process.

If the sawtooth $\mathcal{S}$ is erected on $ab$ or $bc$ (we can assume that it is on $bc$), we proceed as follows. First let us suppose that $ac$ is positive. Let $d$ be a point on $[a,c]$, close enough to $c$ so that $dp_i$ is a positive line segment for every $i$. First we notice that $\mathcal{R}_{a,c}^d$ is licit. Then we erect the sawtooth on $[b,c]$ tooth by tooth using type-$\mathcal{W}$ moves as follows (Figure \ref{le5sfere}).
\begin{itemize}
\item Since $d$ is close to $c$, the triangles $[c,d,p_{m-1}]$ and $[c,p_{m-1},q_m]$ avoid the link; so we can apply the move $\mathcal{W}_{d,c}^{p_{m-1},q_m}$;
\item similarly, for every $i>0$ we can apply the moves $\mathcal{W}_{d,p_i}^{p_{i-1},q_i}$.
\end{itemize}

We now operate the following deformation chain:
\begin{align*}
\mathcal{W}_{d,p_1}^{b,q_1} \circ \dots \circ \mathcal{W}_{d,p_{m-1}}^{p_{m-2},q_{m-1}} \circ \mathcal{W}_{d,c}^{p_{m-1},q_m} \circ \mathcal{R}_{a,c}^d
\end{align*}

Finally, we apply $(\mathcal{E}_{a,b}^d)^{-1}$; since $ad$ and $db$ are positive, this moves increases height by $1$ at most. The result is $L'$. Again, the height is lower than $h(L'')$ during the whole process.

If the edge $ac$ is negative, then necessarily also $ab$ and $bc$ are negative. Without loss of generality, the sawtooth $\mathcal{S}=\mathcal{S}_{p_0,\dots,p_m}^{q_1,\dots,q_m}$ is erected on $bc$. Let us proceed as follows (Figure \ref{le7sfere}).

First, we construct a sawtooth $\mathcal{S}_{r_0,\dots,r_k}^{t_1,\dots,t_k}$ on the edge $ac$, which avoids the convex hull of $\mathcal{S}_{p_0,\dots,p_m}^{q_1,\dots,q_m}$ (this is possible thanks to Lemma \ref{sawexi}, case $1.$). Let $d$ be a point on $[t_k,c]$, close to $c$, so that $dp_i$ is positive for every $i$. We apply a sequence of type-$\mathcal{W}$ moves:
\begin{align*}
\mathcal{W}_{d,p_1}^{b,q_1} \circ \dots \circ \mathcal{W}_{d,p_{m-1}}^{p_{m-2},q_{m-1}} \circ \mathcal{W}_{d,c}^{p_{m-1},q_m}.
\end{align*}

Now, let $x$ be a point on $ac$ close to $c$; let us apply the move $\mathcal{T}_{t_k,d,b}^{t_k,x,b}$, which can be decomposed into type-$\mathcal{R}$ and $\mathcal{W}$ moves, so again the height is preserved. Last we can operate $(\mathcal{S}_{r_0, \dots,r_{k-1},x}^{t_1,\dots,t_k})^{-1}$ (which increases height by $1$) and finally the move $(\mathcal{E}_{a,b}^x)^{-1}$ (which preserves height), thus obtaining the link $L'$. Also in this case, the height is lower than $h(L'')$ during the whole process.
\end{proof}

\begin{proof}[Proof of Markov's Theorem]
Let $L= \hat{\beta}$ and $L'= \hat{\beta'}$. Since $L$ and $L'$ are in braid position, their height is $0$. Let
\begin{align*}
L=L_0 \rightarrow L_1 \rightarrow \dots \rightarrow L_n=L'
\end{align*}
be a chain of deformation moves from $L$ to $L'$. We will prove the theorem by induction on the maximum height appearing in the chain.

As showed in Remark \ref{h0}, if height is constantly $0$ along the whole chain, then the deformation moves which occur in the chain can only be of type $\mathcal{R}$ and $\mathcal{W}$, that means the braids $\beta$ and $\beta'$ are joined by a chain of conjugation and stabilization moves.

Let us suppose that height is not constantly $0$, and let $h_{MAX}>0$ be the maximum height. By Lemma \ref{lemma1} we may suppose that for each $i$ such that $h(L_i)=h_{MAX}$, then $h(L_{i-1})<h_{MAX}$ and $h(L_{i+1})<h_{MAX}$. Indeed, if $h(L_i)=h(L_{i+1})$ then we may join $L_i$ and $L_{i+1}$ by a chain of deformation moves
\begin{align*}
L_i=M_0 \rightarrow M_1 \rightarrow \dots \rightarrow M_k = L_{i+1}
\end{align*}
where $h(M_j) < h(L_i)$ for every $j = 1, \dots k-1$.

Now we can apply Lemma \ref{lemma2} to every $L_{i-1} \rightarrow L_i \rightarrow L_{i+1}$ such that $h(L_i)=h_{MAX}$, thus obtaining subchains
\begin{align*}
L_{i-1}=N_0 \rightarrow N_1 \rightarrow \dots \rightarrow N_l = L_{i+1}
\end{align*}
where $h(N_{j}) < h(L_i)$ for every $j$. In this way we have produced a new chain of deformation moves joining $L$ and $L'$, whose maximum height is lower than $h_{MAX}$. By induction hypothesis we conclude.
\end{proof}

\begin{figure}[p]
\centering
\includegraphics{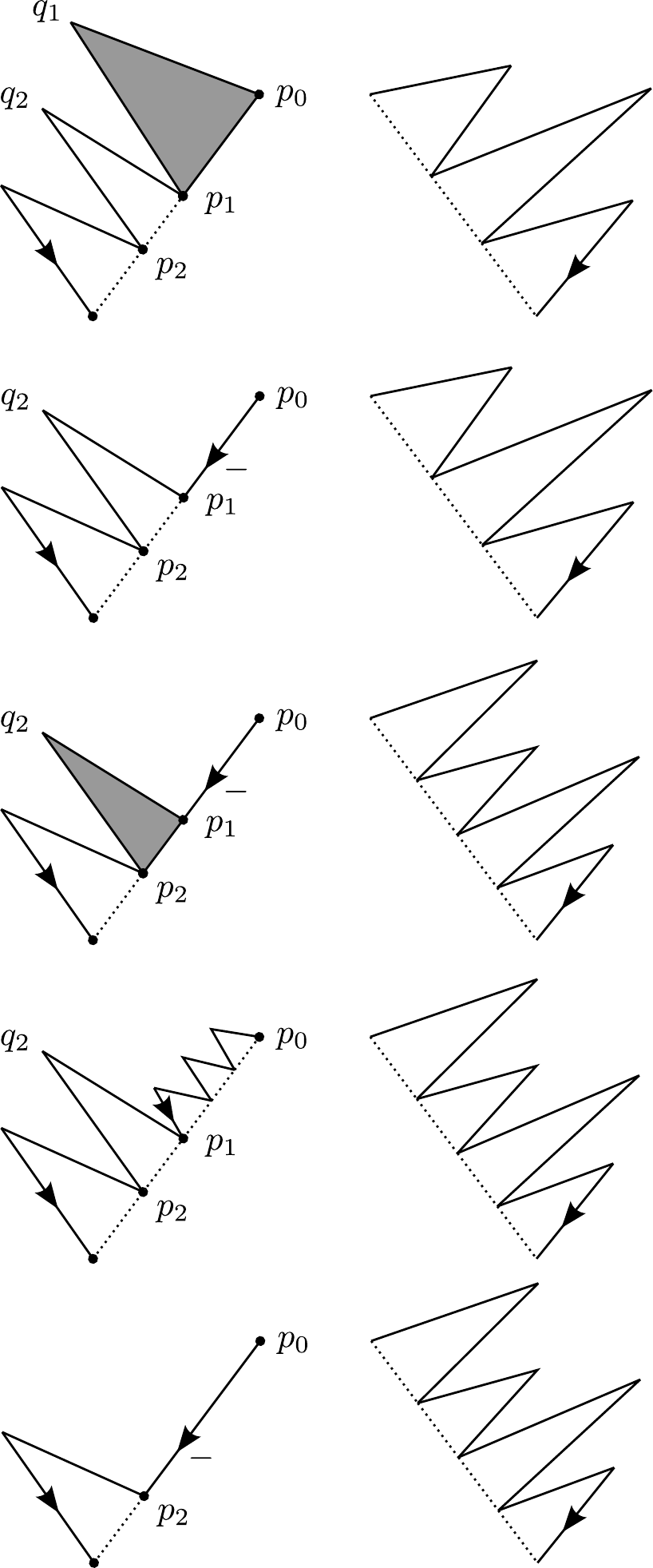}
\caption{Factorization of the chain with two sawteeth on different edges}
\label{labelliamola}
\end{figure}

\begin{figure}[p]
\centering
\includegraphics{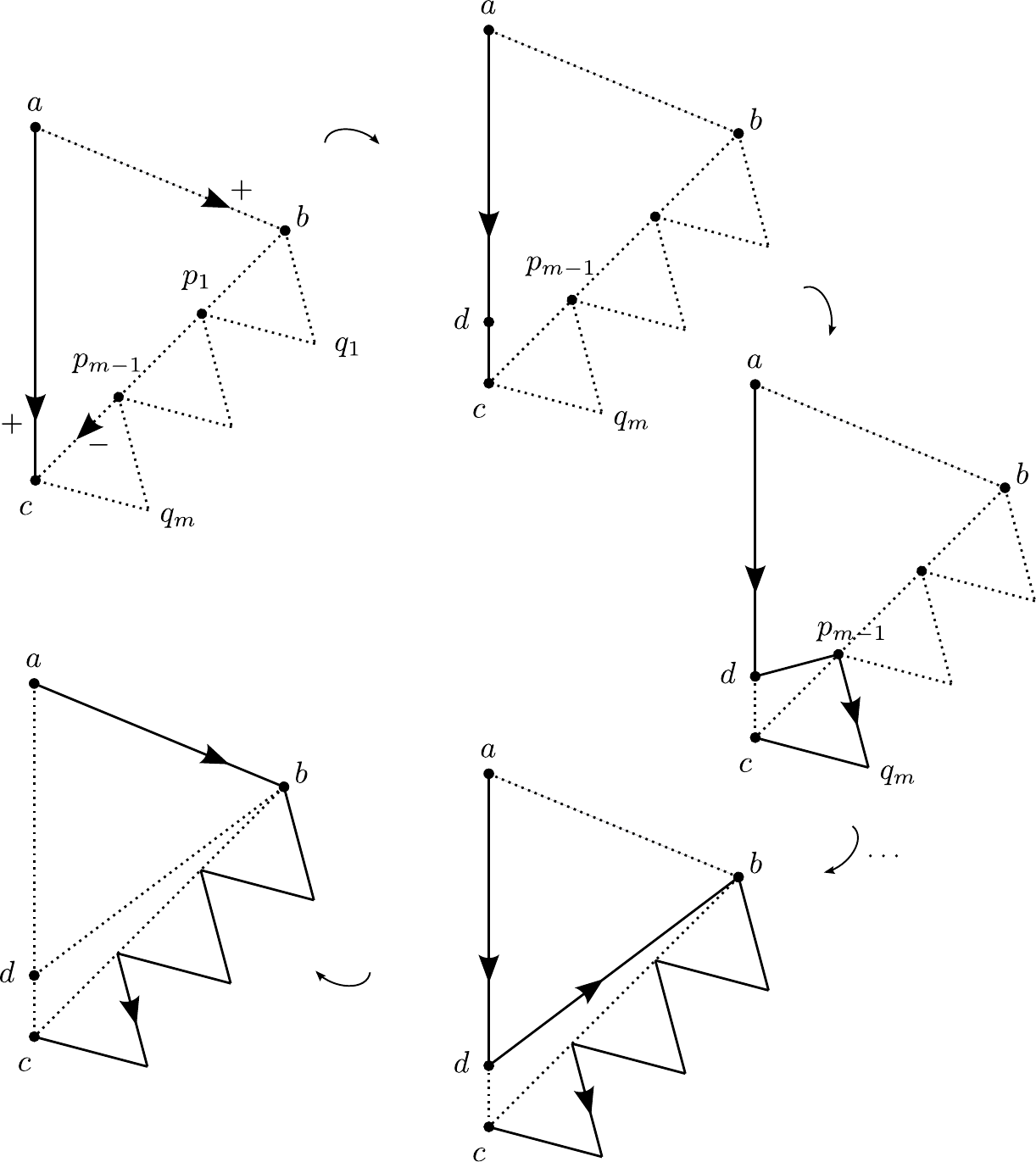}
\caption{Decomposition of the chain when $ac$ is positive}
\label{le5sfere}
\end{figure}

\newgeometry{left=0cm, right=0cm}
\thispagestyle{empty}
\begin{figure}[htb]
\centering
\includegraphics{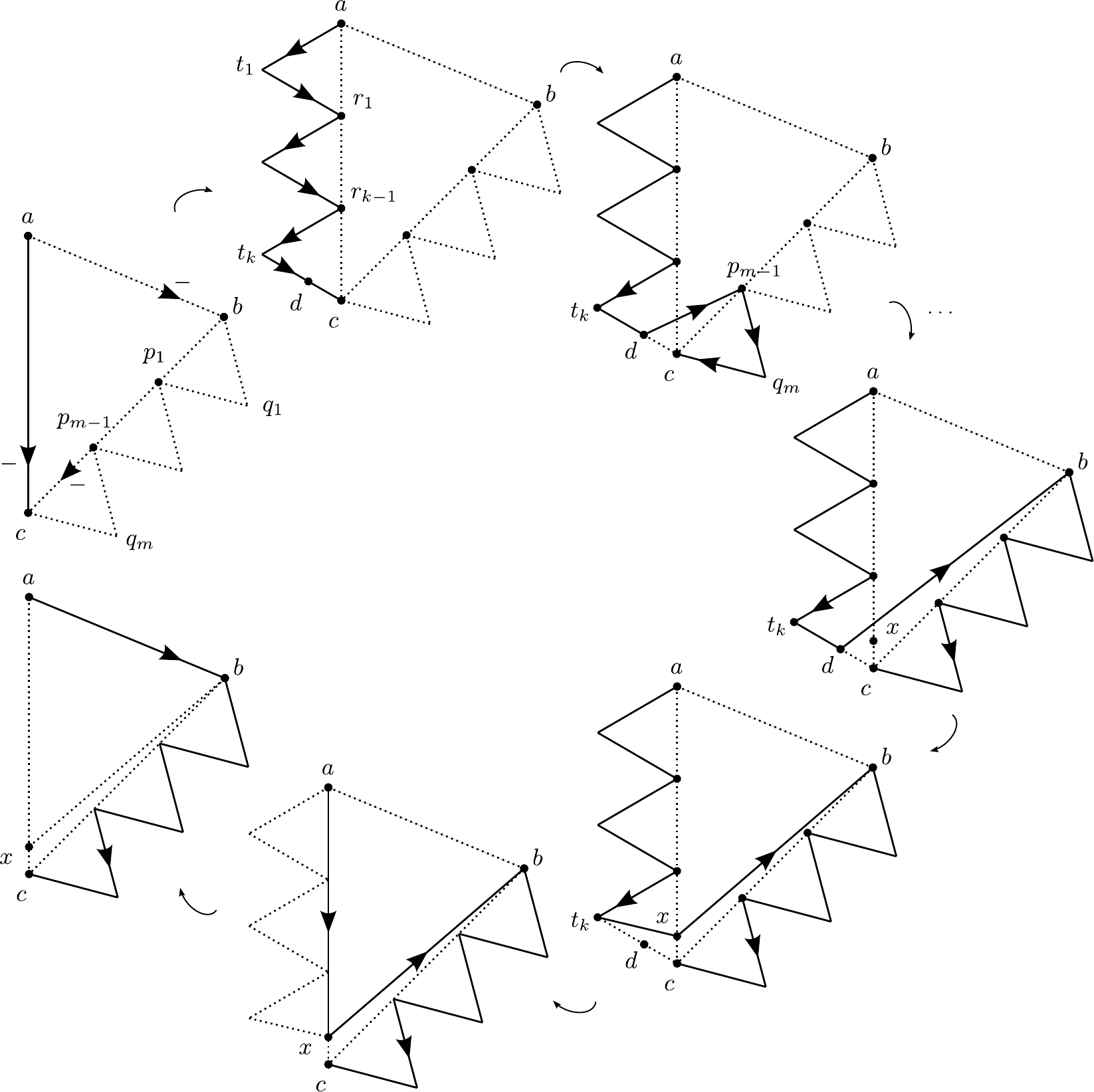}
\caption{Decomposition of the chain when $ac$ is negative}
\label{le7sfere}
\end{figure}
\restoregeometry


\begin{thebibliography}{99}

\bibitem{B} BIRMAN, J., Cannon, J. (1974). Braids, Links, and Mapping Class Groups. (AM-82). Princeton University Press.


\end{thebibliography}
\end{document}